\documentclass[11pt,a4paper]{amsart}
\newcommand{\correctspacing}{\singlespacing}

\usepackage{amsmath,amsthm,amsfonts,amssymb,epsfig}
\usepackage{setspace}
\usepackage{xcolor}
\correctspacing

\usepackage{graphicx}

\usepackage{amssymb}
\usepackage{graphicx}
\usepackage{color}

\usepackage{amssymb}
\usepackage{amsmath}
\usepackage{amsthm}
\usepackage{amscd}
\usepackage{url}

\usepackage{enumerate}

\theoremstyle{plain}
\newtheorem{theorem}{Theorem}[section]
\newtheorem{pro}[theorem]{Proposition}

\newtheorem{lemma}[theorem]{Lemma}
\newtheorem{que}[theorem]{Question}

\theoremstyle{definition}
\newtheorem{definition}[theorem]{Definition}

\theoremstyle{definition}

\newcommand{\M}{\mathcal M} 

\usepackage{bbm}

\renewcommand{\phi}{\varphi}

\newcommand{\Fixme}[1]{\marginpar{}}
\newcommand{\gray}[1]{\marginpar{}} 

 \providecommand{\cM}{\mathcal{M}}

  \providecommand{\RR}{\mathbb{R}}

\providecommand{\cM}{\mathcal{M}} 
 
 \providecommand{\cM}{\mathcal{M}}

\newcommand{\be}{\begin{equation}} \newcommand{\beo}{\begin{equation*}}
\newcommand{\ee}{\end{equation}}    \newcommand{\eeo}{\end{equation*}}
\newcommand{\bea}{\begin{eqnarray}}
\newcommand{\eea}{\end{eqnarray}}
\newcommand{\beao}{\begin{eqnarray*}} \newcommand{\eeao}{\end{eqnarray*}}

\author{Claus Griessler}
\thanks{Financial support through FWF-projects  P21209, P26736, and Y782.}
\date{\today}
\title[Finitely minimal martingale measures]{An extended footnote on finitely minimal martingale measures}

\begin{document}

\maketitle

\begin{abstract}
This note contains a short discussion on the sufficiency of finite optimality in martingale transport. It is shown that finitely minimal martingale measures are solutions of the martingale transport problem when the cost function is upper semi-continuous and bounded from above by  a sum of integrable functions.  As an application a transparent proof of the uniqueness of left-monotone martingale transport plans is given. 
\end{abstract}

\section{Introduction and results}
Given probability measures $\mu$ and $\nu$ on $\RR$, a transport or transport plan $\pi$ is a probability measure on $\RR \times \RR$ that has $\mu$ and $\nu$ as its marginal measures. A martingale transport is any transport $\pi$ that has the martingale property. This means that for each continuous and bounded function $\rho: \RR \rightarrow \RR$ 
\beo
\int \rho (x) (y-x) \, d\pi (x,y) = 0.
\eeo
It is equivalent to require that for the disintegration $( \pi_xÂ )_x$ of $\pi$ one has 
\beo
\int y \, d\pi_x(y) = x \text{ for } \mu \text{-a.e. } x.
\eeo
The set of transport plans is denoted by $\cM (\mu, \nu)$, the set of martingale transport plans  by $\M_{mart}(\mu, \nu)$. \\

For a cost function $c: \RR \times \RR \rightarrow [0, \infty)$, the integral $\int c \, d \pi$ is the  cost of the transport $\pi$. The Monge-Kantorovich problem with a martingale constraint that is considered in this note is to find 

\be{\label{MK}}{\tag{martMK}}
\inf_{\pi \in \M_{mart}(\mu, \nu)} \int c \, d\pi.
\ee
The set $\cM_{\text{mart}}(\mu, \nu)$ is non-empty iff $\mu$ and $\nu$ are in convex order, i.e. $\mu \leq_{cx} \nu$. This relation of the marginals is hence always assumed.\\

Transport problems under martingale constraints and their connections to model-independent finance have been first studied in \cite{HoNe12, BeHePe13, GaHeTo14}. Meanwhile, a rich literature on the subject exists.  Without claiming completeness,  we mention  \cite{AcBePeSc16, BeCoHu14, BeJu14, BeNu14, BoNu15, CaLaMa14, ChKuTa15, DoSo14, GaHeTo14, HeTaTo14, HeTo13, HoKl15, HoNe12, Ju14, St14, TaTo13, Za15}.

As a particular contribution, \cite{BeJu14} started the investigation of finitely minimal martingale measures and the connection of finite minimality to optimality. It is in this context that this note should find its place. 
We repeat the relevant definitions:

\begin{definition}
Let $\alpha$ be a measure on $\RR \times \RR$. A measure $\alpha'$ is called a competitor of $\alpha$ if it has the same marginals as $\alpha$ and 
\beo
\int y \, d\alpha_x (y) = \int y \, d\alpha'_x(y)
\eeo
holds for $p_1(\alpha)$-almost every $x$.
\end{definition}

\begin{definition}
A set $\Gamma$ is called $c$-finitely  optimal if the following holds: whenever $\alpha$ is  a finite measure supported on finitely many points in $\Gamma$, it is cost-minimizing amongst its competitors. I.e., if $\alpha'$ is a competitor of $\alpha$, then
\beo
\int c \, d\alpha \leq \int c \, \alpha'.
\eeo
A martingale transport plan $\pi$  is $c$-finitely optimal if there is a $c$-finitely optimal set $\Gamma$ with $\pi (\Gamma)=1$.
\end{definition}
The idea behind this definition is to bring the concept of $c$-cyclical monotonicity to the martingale problem, and translate relevant results to the martingale case. Cyclical monotonicity of a transport plan in the usual Monge-Kantorovich problem is known to be equivalent to optimality under moderate conditions on the cost functions.  For an overview and the details, readers are referred to \cite{Vil09, AmPr03, ScTe08, BeGoMaSc08, Be15}. 
It is already shown in \cite{BeJu14} that optimal martingale measures are $c$-finitely optimal. The reverse implication  in \cite[Lemma 8.2]{BeJu14} is less satisfactory: it assumes continuity and boundedness of the cost function. This note establishes a sufficiency result in more generality, covering also many types of cost functions not included before. For results of this type the name \emph{monotonicity principle} has been suggested. See for instance Corollary 7.8 in \cite{BeNuTo16}, where a result of similar  spirit is proved.\footnote{The result in \cite{BeNuTo16} shows that if for a given cost function there are marginals $\mu$ and $\nu$ such that the dual problem admits a finite solution, then there is a (finitely minimal) set $\Lambda$ such that any martingale measure $\gamma$ (with arbitrary marginals) concentrated on $\Lambda$ is optimal for the martingale transport problem associated to the marginals of $\gamma$.}
\begin{theorem}
Let $\mu \leq_{cx} \nu$ be probability measures on $\RR$. Let $c: \RR \times \RR \rightarrow [0, \infty)$ be bounded by a sum of integrable functions, i.e. let there be $c_1 \in L_1(\mu)$ and $c_2 \in L_1 (\nu)$ such that 
\beo
c (x,y) \leq c_1 (x) + c_2 (y) 
\eeo
holds for all  $(x,y) \in A \times B$ where $\mu (A) = \nu (B) = 1 $. 
Then any $c$-finitely optimal $\pi \in \M_{mart}(\mu, \nu)$ is optimal. 
\end{theorem}  
A particular  contribution of \cite{BeJu14} was the introduction and investigation of \emph{left-monotone} martingale transports: 
\begin{definition}
A set $\Gamma \subseteq \RR \times \RR$ is called left-monotone if for all\\
 $(x,y^-), (x,y^+), (x',y') \in \Gamma$ with $x<x'$ and $y^- < y^+$ it follows that 
\beo
y' \leq y^- ~~ \text{Â or } ~ y' \geq y^+.
\eeo
A martingale transport $\pi$ is called left-monotone if there is a left-monotone $\Gamma$ with $\pi(\Gamma) = 1$. 
\end{definition}
\cite{BeJu14} shows existence, uniqueness and some remarkable optimality properties of a left-monotone martingale transport $\pi_{lc} \in \M_{\text{mart}}(\mu, \nu)$.  However, the proof of uniqueness seems somewhat opaque as it relies on substantial preparations and a delicate approximation procedure.  Thanks to the wider reach of Theorem 1.3, one can now almost immediately see from the construction of $\pi_{lc}$ that it is the only left-monotone martingale transport. If one accepts Theorem 1.3 as an intuitive fact, this may be seen as a helpful abbreviation. 
\begin{theorem}
The martingale transport $\pi_{lc}$ is the unique left-monotone element in $\M_{\text{mart}}(\mu, \nu)$.
\end{theorem} 

\section{Proof of Theorem 1.3}

A successful strategy in deriving optimality from c-cyclical monotonicity in the usual Monge-Kantorovich problem was to show a splitting-property of c-cyclical monotone sets, see \cite{ScTe08}: if $\Gamma$ is $c$-cyclically monotone then there are measurable functions $\varphi$ and $\psi$ such that
\beo
\varphi (x) + \psi (y) \leq c (x,y)
\eeo 
holds everywhere and with equality on $\Gamma$. Although one does not need to know whether the functions $\varphi$ and $\psi$ are integrable (against $\mu$ and $\nu$, respectively), one can still show that the integral of $\varphi \oplus \psi$ is invariant amongst the measures from $\cM (\mu, \nu)$ for which this integral is defined. Optimality then follows easily.\\ 
As $c$-finite optimality is less than $c$-cyclical monotonicity this attempt cannot be copied verbatim, but it is possibly to show a similar splitting property, namely that, after a suitable reduction of the problem,  there are functions $\varphi,  \psi$, and $\Delta$ such that 
\beo
\varphi (x) + \psi (y) + \Delta (x)\, (y-x) \leq c(x,y).
\eeo
holds everywhere and with equality on the $c$-finitely minimal set given.
The strategy in this note is  to examine the integrability problems that come with this splitting slightly more thoroughly and get the stronger result. This is done in the next definitions and lemmas. We  assume that both $\mu$ and $\nu$ are concentrated on the interval $I$ for which the statements are formulated. 
First we recall a simple result from \cite{BeJu14}:
\begin{pro}
Let $\chi: I \rightarrow \RR$ be convex. Then for any $\pi \in \cM_{\text{mart}}(\mu,\nu)$, the value 
\beo
J(\chi) = \int \int \chi(y)\, d\pi_x(y) - \chi(x) \, d\mu(x)
\eeo
is well-defined in $[0,\infty]$ and does not depend on the choice of $\pi$.
\end{pro}
  \begin{definition}
Let $\varphi, \psi: I \rightarrow \RR$  be measurable functions and $\chi: I \rightarrow \RR$ a convex function such that  both $\int \varphi - \chi \, d \mu < \infty$ and $\int \psi+ \chi \, d \nu < \infty$. Then set
 \beo
 I(\varphi+ \psi) = \int \varphi - \chi \, d \mu + \int \nu + \chi \, d\nu - J(\chi).
 \eeo
 \end{definition}
Note that this definition is, as the notation suggests, independent of the chosen function $\chi$. Indeed, if 
$\chi$ and $\chi'$ are two functions fulfilling the conditions, then we can assume $I(\varphi + \psi, \chi)$ to be finite. Therefore $I(\varphi + \psi, \chi) - I(\varphi + \psi, \chi')$ is defined in $(-\infty, \infty]$. Computing the outcome yields 0, making use of the conditions on $\chi$ and $\chi'$.

\begin{lemma}
Let  $\varphi, \psi, \Delta: I \rightarrow \RR$ be measurable functions and $\pi \in \cM_{\text{mart}}(\mu, \nu)$. If $\varphi$ and $\psi$ are as in definition 2.2 and the function
\beo
\xi(x,y) = \varphi (x) + \psi (y) + \Delta (x) \, (y-x)
\eeo
has a well-defined integral against $\pi$ in $[-\infty, \infty)$, then
\beo
I(\varphi + \psi) = \int \int \varphi (x) + \psi (y) \, d\pi_x(y)\, d\mu(x) = \int \xi \,d\pi.
\eeo
\end{lemma}
\begin{proof}
\begin{align*}
I (\varphi + \psi) &= \int \varphi - \chi \, d \mu + \int \psi + \chi \, d\nu - J(\chi) \\
& = \int \varphi - \chi \, d \mu + \int \psi(y) + \chi(y) \, d\pi_x(y) \, d \mu(x)  - J(\chi) \\
& = \int \Big( \varphi(x) - \chi(x)  + \int \psi(y) + \chi(y)  \, d\pi_x(y)  - \int \chi (y') \, d\pi_x(y') + \chi(x) \Big)\, d\mu (x)\\
&= \int \Big( \varphi(x)  + \int \psi (y) \, d \pi_x(y) \Big) \, d \mu(x)\\
&= \int \int \varphi (x) + \psi (y) \, d\pi_x(y) \, d \mu(x)\\
&= \int \xi d \pi.
\end{align*}
In the last step the martingale property of $\pi$ was used.
\end{proof}

\begin{lemma}
Let  $\varphi, \psi: I \rightarrow \RR$ and $\Delta: I \rightarrow \RR$ be  functions, and let $c_1 \in L_1 (\mu)$ and $c_2 \in L_1(\nu)$  be such that the inequality
\beo
\varphi (x) + \psi (y) + \Delta (x) \; (y-x) \leq c_1(x) + c_2(y) 
\eeo 
holds on a set $A \times B$ with $\mu (A) = \nu (B) = 1$.
Then there is an interval $I' \subseteq I$ with $\mu(I') = \nu(I') = 1$, and a convex function $\chi: I' \rightarrow \RR$ such that $\varphi-\chi \leq c_1 ~\mu-a.e.$  and $\psi+ \chi \leq c_2 ~\nu-a.e.$
\end{lemma}
\begin{proof}
On $I$ define the functions  
\beo
f  = \varphi - c_1
\eeo 
and
\beo 
g  = \psi  - c_2.
\eeo
Define $\chi$ to be the largest convex function on $I$ that is smaller than $-g$ on $B$, i.e.  
\beo
\chi (x) = (-g)^{**}(x) = \sup \bigl\{ \alpha (x): ~ \alpha ~\text{ convex } I \rightarrow \RR ~ \text{ with } \alpha \leq -g ~\text{ on } B \bigr\}.
\eeo
For any $x\in A$ the  function $y \mapsto f(x) + \Delta (x) (y-x)$ is convex and not larger than $-g$ on $B$, hence $\chi > -\infty$. Also, $\chi < \infty$ on $I' = \text{conv }B$. We have $\nu (I') =1$, and due to the convex order also $\mu (I') =1$.
 For all $x\in A$ and all $y \in \text{conv } B$, 
\beo
f(x) + \Delta (x) (y-x) \leq \chi (y).
\eeo
Hence, setting $y=x$, for all $x\in A\cap I'$, and thus $\mu$-a.e., $f(x) \leq \chi( x)$. 
Therefore, 
\beo
\varphi - \chi \leq c_1 ~ \mu-\text{a.e.}
\eeo
and 
\beo
\psi + \chi \leq c_2 ~\nu-\text{a.e.}
\eeo
\end{proof}\textbf{}

\begin{proof}[Proof of Theorem 1.3.] 
We first recall that it is enough to show the claim when the marginals are irreducible on an interval $I$: 
this  means that $\mu (I) = \nu( \bar{I}) =  1$, and that the potential function $u_{\mu}$ of $\mu$ is strictly smaller than the potential function $u_{\nu}$ of $\nu$ on $I$. 
That this does not lead to a loss of generality follows from Theorem 8.4 in \cite{BeJu14}: $\mu$ and $\nu$ can be decomposed into $\sum_{k\geq 0} \mu_k$, $\sum_{k\geq 0} \nu_k$ such that $\mu_0= \nu_0$ and $\mu_k$ and $\nu_k$ are irreducible on $I_k$, where the intervals $I_k$ are the components of $\{ u_{\mu} < u_{\nu} \}$ and $\mu_k = \mu|_{I_k}$. Furthermore, any martingale transport between $\mu$ and $\nu$ can then be decomposed into a sum of martingale transports between $\mu_k$ and $\nu_k$. \\
The irreducibility assumption can be used to show that a $c$-finitely optimal set $\Gamma$ can be w.l.o.g. assumed to  have additional regularity properties, and it can then be shown that there are measurable functions $\varphi, \psi: \RR \rightarrow [-\infty, \infty)$ and $\Delta: \RR \rightarrow \RR$ such that 
\beo
\varphi (x) + \psi (y) + \Delta (x) (y-x) \leq c(x,y)
\eeo
holds everywhere and with equality on $\Gamma$. This is Proposition 8.10 in \cite{BeJu14}.\footnote{The result in \cite{BeJu14} is only formulated for continuous cost functions, but it is not difficult to see that their proof also works for upper semi-continuous cost functions.}
\\
Assume now that $\pi$ is a $c$-finitely optimal transport plan, that is concentrated on a $c$-finitely optimal set $\Gamma$ such that there are functions $\varphi, \psi, \Delta$ as above. After suitably restricting and changing the functions $\varphi$ and $\psi$ we are in the situation of Lemma 2.4, find a convex function $\chi$ as in Definition 2.2, and can define $I(\varphi + \psi)$. 
We set 
\beo
\xi (x,y) = \varphi(x) + \psi (y) + \Delta (x) (y-x).
\eeo
Due to the assumption on $c$ being bounded by a sum of integrable functions, $\xi$ is integrable for all transport plans. If $\pi' \in \cM_{\text{mart}} (\mu, \nu)$, then due to Lemma 2.3
\beo
I(\varphi+ \psi) = \int \xi \, d\pi
\eeo
and
\beo
I(\varphi + \psi) = \int \xi \, d\pi'.
\eeo
We also have $\int \xi \, d \pi = \int c \, d\pi$ and $\int \xi \, d\pi'\leq \int c \, d\pi'$, therefore
\beo
\int c \, d\pi \leq \int c \, d\pi'.
\eeo
\end{proof}

\section{Proof of Theorem 1.5}

Before applying Theorem 1.3 to prove the claim, we recall how a particular (and as it turns out, \emph{the}) left-monotone coupling $\pi_{lc}$ is constructed. From this construction (or a characterization implied by it), Theorem 1.3 swiftly  leads to uniqueness. 

A measure $\mu$ is smaller in the extended convex order than a measure $\nu$ (notation: $\mu \leq_E \nu$) if for each nonnegative convex function $\varphi$
\beo
\int \varphi \, d\mu \leq \int \varphi \, d\nu.
\eeo
For such measures there is a measure $\theta$ which is, (1), smaller than $\nu$ (i.e. $\theta \leq \nu$), but, (2), still larger than $\mu$ in the convex order (i.e. $\mu \leq_{cx} \theta$), and, (3),  smallest in convex order among all measures having the first two properties. $\theta$ is then called the shadow of $\mu$ in $\nu$, notation: $S^{\nu} (\mu)$. Existence of the shadow is Lemma 4.6 in \cite{BeJu14}.
The coupling $\pi_{lc}$ is constructed via the following rule: 
let  $\pi_{lc}$ be the measure that transports, for each $s\in \RR$, the restriction $\mu|_{(-\infty, s]}$ to its shadow in $\nu$, i.e. to $S^{\nu}(\mu|_{(-\infty,s]})$. That this prescription leads indeed to a left-monotone measure in $\cM_{\text{mart}}(\mu, \nu)$ is Theorem 4.18 in \cite{BeJu14}. Most important for us is an easy consequence of this: letting $c_{s,t}$ denote the cost function $(x,y) \mapsto \mathbbm{1}_{(-\infty, s]}(x) |y-t|$, the measure $\pi_{lc}$ is the unique measure in $\cM_{\text{mart}}(\mu, \nu)$ that is a minimizer for all the functions $c_{s,t}$. This follows easily from the definition of the shadow and the fact that the functions $y\mapsto |y-t|$ are enough to describe the convex order of measures on $\RR$. 
\begin{lemma}
A left-monotone set $\Gamma$ is $c_{s,t}$-finitely minimal for all $s,t \in \RR$.
\end{lemma}
\begin{proof}
Let $\alpha$ be a finite measure on finitely many points in $\Gamma$, let $\alpha'$ be a competitor of $\alpha$ and let $x_1 < \dots < x_n$ be the points on which $p_1(\alpha) = p_1(\alpha')$ is concentrated. If $s< x_1$ or $s\geq x_n$, then $\int c_{s,t} \, d\alpha = \int c_{s,t} \, d\alpha'$. In slight misuse of notation, writing $\alpha_{x_i}$ for the one-dimensional measure $\alpha |_{\{x_i\} \times \RR}$ (i.e. the non-normalized disintegration of $\alpha$), we claim that 
\beo
\alpha_{x_1} + \dots + \alpha_{x_i} \leq_{cx} \alpha'_{x_1} + \dots + \alpha'_{x_i}
\eeo
holds for all $i= 1,  \dots, n$. This in particular implies  that $\int c_{s,t} \,  d\alpha \leq \int c_{s,t} \, d \alpha' $ also holds for all $s \in [x_1, x_n)$:\\
Let $i=1$: set $y_1^- = \min \{ y: \alpha (x_1, y) > 0 \}$ and $y_1^+ = \max \{ y: \alpha (x_1, y) > 0 \}$. 
By left-monotonicity, for all $y \in (y_1^-, y_1^+)$ we  must have 
$$
\alpha (x_1, y) = p_1(\alpha) (y) \geq \alpha' (x_1,y).
$$
But $\int y \, d \alpha_{x_1}(y) = \int y \, d\alpha'_{x_1}(y)$ and $\sum_y \alpha_{x_1}(y) = \sum_y \alpha'_{x_1}(y) $, hence 
$$
\alpha_{x_1} \leq_{cx} \alpha'_{x_1}.
$$
Let $i=2$: 
If $y_2^- = \min \{ y: \alpha(x_2, y) > 0  \} \geq y_1^+$ or if
$y_2^+ = \max \{ y: \alpha (x_2, y) > 0 \} \leq y_1^-$, then one can repeat the above argument for $\alpha_{x_2}, \alpha'_{x_2}$ and the claim follows by addition. Elsewise, 
one can repeat the argument for $\alpha_{x_1}+\alpha_{x_2}, \alpha'_{x_1}+\alpha'_{x_2}$ in place of $\alpha_{x_1}, \alpha'_{x_1}$. \\
For larger $i$ only the notation is more complicated.
\end{proof}

\begin{proof}[Proof of Theorem 1.5]
Let $\pi \in \cM_{\text{mart}}(\mu, \nu)$ be left-monotone and $\Gamma$ be a left-monotone set with $\pi (\Gamma) = 1$. By the above lemma, $\Gamma$ is $c_{s,t}$-finitely minimal for all $s,t \in \RR$. By Theorem 1.3 it is hence a minimizer for all $c_{s,t}$, and it must therefore be equal to $\pi_{lc}$. 
\end{proof}

\bibliographystyle{alpha}
\bibliography{joint_biblio}

\end{document}